\documentclass[12pt]{article}
\RequirePackage{fix-cm}
\usepackage[T1]{fontenc}
\usepackage[utf8]{inputenc}
\usepackage{lmodern}
\usepackage{textcomp}
\usepackage{microtype}
\usepackage{authblk}
\usepackage{amsmath}
\usepackage{amsfonts}
\usepackage{amssymb}
\usepackage{amsthm}
\usepackage{latexsym}
\usepackage{mathtools}
\usepackage{multirow}
\usepackage{array}
\usepackage{tabulary}
\usepackage{cases}

\usepackage{geometry} 
\usepackage[toc,page]{appendix}
\usepackage{lineno} 

\usepackage{xcolor}
\usepackage{float} 
\usepackage{placeins} 

\usepackage{csquotes}
\usepackage{url}
\usepackage{natbib}
\setcitestyle{numbers,open={[},close={]}}
\newcommand{\textcite}{\cite}

\usepackage{graphicx}
\usepackage{setspace}

\usepackage{thmtools}

\declaretheorem[
	name=Theorem
	]{theorem}
\declaretheorem[
	name=Lemma,
	]{lemma}

\newcommand{\ga}{\gamma}

\newcommand{\rh}{\rho}

\newcommand{\ze}{\zeta}

\newcommand\PinfNB{P_{\rm{NB}}^\infty}

\newcommand\fNB{f_{\rm{NB}}}
\newcommand\gNB{g_{\rm{NB}}}

\newcommand\varphiFL{\varphi_{\rm{FL}}}

\newcommand\varphiNB{\varphi_{\rm{NB}}}

\newcommand\rNB{\rh_{\rm{NB}}}
\newcommand\pNB{\pi_{\rm{NB}}}

\newcommand\zet{{\tilde\ze}}

\newcommand\yx{y(x)}
\newcommand\cgt{c_{\tilde g}}
\defcitealias{Agresti1974}{Agresti's (1974)}
\defcitealias{Athreya1992}{Athreya's (1992)}
\defcitealias{Haldane1927}{Haldane's (1927)}
\defcitealias{Hill1982a}{Hill's (1982)}
\defcitealias{Hoellinger2019}{H\"ollinger's (2019)}
\defcitealias{Lande1976}{Lande's (1976) equation}
\defcitealias{Mathematica}{{\it Mathematica}}
\defcitealias{Quine1976}{Quine's (1976)}
\defcitealias{Pollak1971}{Pollak's (1971)}

\begin{document}
\normalbaselineskip18pt
\baselineskip18pt
\global\hoffset=-10truemm
\global\voffset=-10truemm
\allowdisplaybreaks[2]
\setcounter{page}{1}
%
\title{\bf Addendum to:\\ Bounds for survival probabilities in \\ supercritical Galton-Watson processes and \\ applications to population genetics 
\vphantom{T}
}
\author{\bf Reinhard B\"urger}

\affil{Faculty of Mathematics \\ University of Vienna \\ 1090 Vienna, Austria  \\
\vspace{0.5cm} 
Email: reinhard.buerger@univie.ac.at \\
{\small ORCID: 0000-0003-2174-4968}
\date{}
}
%

\maketitle

\vspace{0.2cm}


\noindent
{\bf Abstract.} In this addendum we extend Theorem 4.6 on the negative binomial distribution in \cite{Buerger2026}. We  prove that the fractional linear lower bound to the negative binomial generating function derived there is indeed valid for every $x\in[0,1]$, and not only for $x\in[0,\PinfNB]$, where $\PinfNB$ is the extinction probability of the associated Galton-Watson process.

\vspace{2.2cm}


\vspace{0.1cm}

\newpage

We assume that the reader is familiar with the content and notation in \citep{Buerger2026}. Here is the generalized version of Theorem 4.6 in \citep{Buerger2026}.

\begin{theorem}\label{thm:NB_ge_FL}
For every $r\ge2$ and every $\ze\in(0,1)$, the pgfs $\varphiNB$ and $\varphiFL$ satisfy
\begin{linenomath}\begin{equation}\label{varphiNB_ge_varphiFL}
	\varphiFL(x;\pNB,\rNB) \le \varphiNB\Bigl(x;r,\frac{\ze(1-\ze^r)}{1-\ze^{r+1}}\Bigr) \;\text{ for every }x\in[0,1]\,.
\end{equation}\end{linenomath}
Equality holds if and only if $x = \PinfNB$ or $x=1$.
\end{theorem}

In the following, we first repeat with trivial adaptations the proof of the above theorem for $x\in[0,\PinfNB]$, as given in Appendix B in \cite{Buerger2026}. In particular, we retain the equation numbering from \cite{Buerger2026} until eq.~(B.11). Then the new proof starts.

We recall that $\ze^r=\PinfNB \in(0,1)$ and define
\begin{linenomath}\begin{equation}\label{zet_y_defs}
	\zet = \sum_{k=0}^{r-1} \ze^k \; \text{ and }\; \yx = \frac{\ze^r-x}{\zet} \,. \tag{B.1}
\end{equation}\end{linenomath}
Using eq.~(A.2) in \cite{Buerger2026}, we observe that
\begin{equation}\label{y_bounds}
	y(0)=\frac{\ze^r}{\zet} < \frac{1}{r}\ze^{(r+1)/2}<\frac{1}{r} \,, \; y(\ze^r)=0 \,, \; y(1) = \frac{\ze^r-1}{\zet} = -(1-\ze) \,.  \tag{B.2}
\end{equation}
With these abbreviations, we can express the negative binomial generating function $\varphiNB$ (defined in eq.~(4.23) in \citep{Buerger2026}) and the (prospective) bounding fractional linear generating function $\varphiFL$ (defined in eq.~(2.9) in \citep{Buerger2026}), where $\pNB$ and $\rNB$ are given in eq.~(4.25) in \cite{Buerger2026}, as follows:
\begin{linenomath}\begin{equation}
 \varphiNB\Bigl(x;r,\frac{\ze(1-\ze^r)}{1-\ze^{r+1}}\Bigr) = \ze^r\,\left(\frac{1}{1+\yx} \right)^r  \tag{B.3}
\end{equation}\end{linenomath}
and
\begin{linenomath}\begin{equation}
	\varphiFL(x;\pNB,\rNB) = \ze^r\,\frac{1-x - r\yx}{1-x - r\ze^r \yx}\,. \tag{B.4}
\end{equation}\end{linenomath}
We note that numerator and denominator are always positive.

We define
\begin{equation}\label{def_fNB}
	\fNB(x;r,\ze) := \ze^r \left(1/\varphiFL(x;\pNB,\rNB) -  1/\varphiNB\Bigl(x;r,\frac{\ze(1-\ze^r)}{1-\ze^{r+1}}\Bigr) \right) \tag{B.5}
\end{equation}
and
\begin{equation}\label{def_gNB}
	\gNB(y(x);r,\ze):=\fNB(x;r,\ze)\frac{1-x-r y(x)}{(1-\ze)^2} \,, \tag{B.6}
\end{equation}
where $1-x-r y(x)>0$.

\begin{proof}[{Proof of the inequality \eqref{varphiNB_ge_varphiFL} if $x\in[0,\PinfNB]$ (this was already proved in \cite{Buerger2026})} ] 
Proving this inequality is equivalent to showing that $\gNB(y(x);r,\ze)>0$ for $0<y(x)<\ze^r/\zet$, $r\ge2$, and $0<\ze<1$. 
In the following we write $y=y(x)$. Using the transformation $x= \ze^r - y\frac{1-\ze^r}{1-\ze}$, we obtain 
\begin{equation}\label{def_gNB_y}
	\gNB(y;r,\ze) = \fNB\Bigl(\ze^r - y\frac{1-\ze^r}{1-\ze};r,\ze\Bigr) \frac{(1-\ze)(1-\ze^r)+y(1-\ze^r-r(1-\ze))}{(1-\ze)^3}\,, \tag{B.7}
\end{equation}
which we can rewrite as
\begin{equation} 
	\gNB(y) 
	 = \frac{y\bigl((1+y)^r-1\bigr)\bigl(r(1-\ze)-(1-\ze^r)\bigr) - (1-\ze)(1-\ze^r) \Bigl(\bigl((1+y)^r-1\bigr) - ry \Bigr)}{(1-\ze)^3} \,.  \tag{B.8}
\end{equation}

By the binomial expansion $(1+y)^r-1 = y\sum_{j=0}^{r-1}\binom{r}{j+1}y^j$ and after collection of coefficients of $y^j$, we obtain
\begin{subequations}
\begin{align}
	\gNB(y) &= \frac{y^2}{(1-\ze)^2} \sum_{j=0}^{r-1} y^j\left( \binom{r}{j+1}\Bigl( r-\frac{1-\ze^r}{1-\ze}\Bigr) - \binom{r}{j+2}(1-\ze^r) \right) \tag{B.9a} \\
	&= y^2\sum_{j=0}^{r-1} y^j \binom{r}{j+1}\frac{1}{(1-\ze)^2}\left(\Bigl( r-\frac{1-\ze^r}{1-\ze}\Bigr) - \frac{r-j-1}{j+2}(1-\ze^r) \right) \,. \tag{B.9b}
\end{align}
\end{subequations}
Finally, expansion in terms of $\ze$ yields after appropriate rearrangement
\begin{equation}\label{gNB_final}
	\gNB(y;r,\ze) = y^2 \sum_{j=0}^{r-1} y^j \binom{r}{j+1} c_g(r,j,\ze) \,, \tag{B.10}
\end{equation}
where
\begin{equation}\label{c_g}
	c_g(r,j,\ze) = \frac{1}{2(j+2)}\left( \sum_{k=0}^{r-2} \ze^k(k+1)[2r(1+j)-(2+j)k-2] + \frac{\ze^{r-1}}{1-\ze}r(r+1)j  \right)\,.  \tag{B.11}
\end{equation}
Because $2r(1+j)-(2+j)k-2 \ge 2 + j(2r-k) \ge 2$, we obtain $c_g(r,j,\ze) > 0$ for every $0\le j\le r-1$ and every $\ze\in(0,1)$. Therefore, $\gNB(y;r,\ze) >0$ if $y>0$ and $\fNB(x;r,\ze)>0$ if $0\le x < \ze^r$.
\end{proof}

\setcounter{equation}{1}

\begin{proof}[{Proof of the inequality \eqref{varphiNB_ge_varphiFL} if $x\in(\PinfNB,1]$} ] 
We use the transformation $y=u-(1-\ze)$ which, by \eqref{y_bounds}, is equivalent to $u = (1-x)\frac{1-\ze}{1-\ze^r}$. Then, $0\le u \le 1-\ze$ if $\PinfNB = \ze^r \le x \le 1$. 
We use series expansion of 
\begin{equation}\label{def_tilde_gNB_u}
	\tilde g_{\rm NB}(u;r,\ze) := \frac{1-\ze}{(u-(1-\ze))^2}\gNB(u-(1-\ze);r,\ze) 
\end{equation}
in $u$ and $\ze$ to prove that $\tilde g_{\rm NB}(u;r,\ze)\ge0$ if $0\le u \le 1-\ze$ and $0<\ze<1$. The key element of the proof is the subsequent lemma.

\begin{lemma}\label{lem:cgt}
The following representation holds:
\begin{equation}
	\tilde g_{\rm NB}(u;r,\ze) = \sum_{k=1}^{r-1}\sum_{n=0}^{2r-3-k} u^k \ze^n \cgt(r,k,n) \,,
\end{equation}
where
\begin{subnumcases}
	{\cgt(r,k,n)  = \label{cgt_final}} \textstyle
		k \binom{n+k+2}{k+2}  & if $0\le n \le r-1-k$\,, \vspace{2mm} \label{cgt_final_a} \\
		\textstyle
		\binom{r}{k+1}\frac{(2r-k)(k+1)-(n+1)(k+2)}{k+2} & if $r-k \le n \le r-1$\,, \vspace{2mm} \label{cgt_final_b} \\
		\textstyle
		\binom{r}{k+1}\frac{(2r-k)(k+1)-(n+1)(k+2)}{k+2} \vspace{2mm} \notag \\
		\textstyle
		\quad+ \binom{k+(n-r)+1}{k+1}\frac{(2r-k)(k+1)-(n+1)k}{k+2} & if $r \le n \le 2r-3-k$ \,. \label{cgt_final_c} 
\end{subnumcases}
\end{lemma}

We will prove this lemma below. Obviously, the expression in \eqref{cgt_final_a} is positive.
The expression $(2r-k)(k+1)-(n+1)(k+2)$ in \eqref{cgt_final_b} and \eqref{cgt_final_c} is positive if $ r-k \le n \le r-1$. However, it can be negative if $r \le n \le 2r-k-3$. 

To finish the proof of inequality \eqref{varphiNB_ge_varphiFL} if $\PinfNB < x <1$, it remains to show that the sum of the expressions in \eqref{cgt_final_c} is positive. We start by noting that this case applies only if $r\ge4$ and $k\le r-3$. 
Two special cases emerging readily from \eqref{cgt_final_c} are
\begin{align*}
	\cgt(r,k,r) &= \binom{r}{k+1}\frac{(r-k-2)k-2}{k+2} + r-k \,, \\
	\cgt(r,k,2r-3-k) &= \binom{r}{k+1}\frac{(k+1)k}{r(r-1)} = \binom{r-2}{k-1} \,.
\end{align*}
Both are positive because $1 \le k \le r-3$ and $r\ge4$. We note in passing that $\cgt(r,k,2r-i-k)=0$ if $i\in\{0,1,2\}$.

We observe that the term $c_{\tilde g,a}(r,k,n):=\binom{r}{k+1}\frac{(2r-k)(k+1)-(n+1)(k+2)}{k+2}$ is (linearly) decreasing in $n$ (and becomes negative for large $n)$.
Straightforward algebra shows that the term $c_{\tilde g,b}(r,k,n):=\binom{k+(n-r)+1}{k+1}\frac{(2r-k)(k+1)-(n+1)k}{k+2}$ is positive and increasing in $n$ under the given constraints. Moreover, the second-order forward difference is computed to 
\begin{equation}
	c_{\tilde g,b}(r,k,n+2)-c_{\tilde g,b}(r,k,n)-2c_{\tilde g,b}(r,k,n+1) = \binom{k+(n-r)+1}{k}\frac{(2r-n-k-2)k}{n-r+2}\,.
\end{equation}
This is positive if $r\le n \le 2r-k-3$, equals 0 if $n=2r-k-2$, and is negative for larger $n$.
We obtain that (i) $c_{\tilde g,b}(r,k,n)$ is convex, increasing, and positive if $r\le n \le 2r-k-3$, (ii) $c_{\tilde g,a}(r,k,n)$ is linear (and decreasing), (iii) the first-order differences of $c_{\tilde g,b}$ reach their maximum at $n=2r-k-2$, where they equal $\binom{r}{k+1}$, which is the negative of the first-order difference of $c_{\tilde g,a}$, and (iv) $\cgt(r,k,r)>0$ and $\cgt(r,k,2r-k-3)>0$. Therefore, we conclude that $\cgt(r,k,n)>0$ holds in this range.
 
As consequence, $\tilde g_{\rm NB}(u;r,\ze)$ is strictly positive if $0<u<1-\ze$. By \eqref{def_tilde_gNB_u} and \eqref{def_gNB_y}, this implies the desired positivity of $\fNB(x)$ if $\PinfNB < x <1$.
Pending the proof of Lemma \ref{lem:cgt}, this finishes the proof of inequality \eqref{varphiNB_ge_varphiFL} for $x\in(\PinfNB,1]$ and hence the proof of Theorem \ref{thm:NB_ge_FL}.
\end{proof}

Of course, the expressions given in Lemma \ref{lem:cgt} yield those in eq.~(12) in \cite{Buerger2026} for $r=2,3,4,5$.

\begin{proof}[Proof of Lemma \ref{lem:cgt}]
It follows readily from the definition of $\tilde g_{\rm NB}$ and that of $g_{\rm NB}$ that only the coefficients listed are nonzero, i.e., no higher-order terms can occur. The proof is structured in several steps. 

STEP 1. We derive the representation
\begin{subequations}\label{cgt_all}
\begin{equation}\label{cgt}
	\cgt(r,k,n) = \begin{cases} \sum\limits_{m=0}^{r-1-k} \binom{m+k}{k} \binom{r}{m+k+1} c_{\tilde g_1}(r,k,n,m)  &\text{ if } 0 \le n \le r-1\,, \vspace{3mm} \\
	  \sum\limits_{m=n-r+1}^{r-1-k} \binom{m+k}{k} \binom{r}{m+k+1} c_{\tilde g_2}(r,k,n,m)  &\text{ if } r \le n \le 2r-3-k\,,	\end{cases}
\end{equation}
where,
\begin{equation}\label{cgt1}
	c_{\tilde g_1}(r,k,n,m) = \begin{cases} (-1)^m \Bigl(r - \frac{r+1}{m+k+2} \Bigr) &\text{ if } n=0	\,,  \vspace{2mm}  \\
	(-1)^{m-n}\biggl[ \binom{m-1}{n} \Bigl(r - \frac{r+1}{k+m+2} \Bigr)  + \binom{m-2}{n-1} \biggr] &\text{ if } 1 \le n \le r-1\,,\end{cases} 
\end{equation}
and
\begin{equation}\label{cgt2}
	c_{\tilde g_2}(r,k,n,m)=\begin{cases} 1- \frac{r+1}{k+3} &\text{ if  $n=r$  and  $m=1$}\,,  \vspace{2mm}  \\ 
		(-1)^m \frac{r+1}{k+2+m} &\text{ if  $n=r$  and  $m>1$} \,, \vspace{2mm}  \\
		 (-1)^{m-(n-r)}\Bigl[\binom{m-1}{n-r} \frac{r+1}{k+m+2} - \binom{m-2}{n-r-1}\Bigr]&\text{ if } r+1 \le n \le 2r-3-k \,. \end{cases}
\end{equation}
\end{subequations}
We note that $c_{\tilde g_1}(r,k,n,0) = r-n-\frac{r+1}{k+2}$ because, in case $n\ge1$, we have $(-1)^{0-n}\binom{0-2}{n-1} = -n$. 
The coefficients $c_{\tilde  g_2}(r,k,n,m)$ vanish if $n\ge 2r-2-k$ because then the summation in the second case of \eqref{cgt} includes only $m=r-1-k$ and the binomial coefficients $\tbinom{m-1}{n}$ and $\tbinom{m-2}{n-1}$ in \eqref{cgt2} vanish because $m-1< n$. 

To prove \eqref{cgt_all}, we start with \eqref{gNB_final}. From the definition of $\tilde g_{\rm NB}(u;r,\ze)$ and $y=u-(1-\ze)$, we obtain
\begin{equation}
	\tilde g_{\rm NB}(u;r,\ze) = \sum_{k=0}^{r-1} u^k \left(\sum_{j=k}^{r-1} \binom{j}{k}(\ze-1)^{j-k} \frac{1}{1-\ze} c_g(r,j,\ze) \right)\,,
\end{equation}
where $c_g(r,j,\ze)$ is given in \eqref{c_g}. 

Binomial expansion and rearrangement yields (for given $r$ and $k$, and with $i=j-k$):
\begin{align}\label{sum_j}
	&\sum_{j=k}^{r-1} \binom{j}{k}(\ze-1)^{j-k} \frac{1}{1-\ze} c_g(r,j,\ze) \notag \\
	&\quad = \sum_{i=0}^{r-1-k} \binom{i+k}{k} \sum_{n=0}^{r-1+i} \ze^n \notag  \\
	&\qquad \times \sum_{l=\max\{0,n-i\}}^{\min\{r-1,n\}} (-1)^{i-n+l} \tbinom{i}{i-n+l} \biggl[\tbinom{r}{i+k+1}(r-l)-\tbinom{r+1}{i+k+2} \biggr] \,.
\end{align}
To simply the last sum (over $l$) we need to compute the following sums:
\begin{equation*}
	\ga_1(n,i) = \sum_{l=\max\{0,n-i\}}^{n} (-1)^{i-n+l} \tbinom{i}{i-n+l} \,, 
		\quad \ga_2(n,i) = \sum_{l=\max\{0,n-i\}}^{n} (-1)^{i-n+l} \tbinom{i}{i-n+l}l
\end{equation*}
if $0\le n \le r-1$, and 
\begin{equation*}
	\ga_3(r,n,i) = \sum_{l=\max\{0,n-i\}}^{r-1} (-1)^{i-n+l} \tbinom{i}{i-n+l} \,, 
		\quad \ga_4(r,n,i) =  \sum_{l=\max\{0,n-i\}}^{r-1} (-1)^{i-n+l} \tbinom{i}{i-n+l}l
\end{equation*}
if $r \le n \le r-1 + i$\,.


Straightforward algebra yields
\begin{subequations}
\begin{align}
	\ga_1(n,i) & = (-1)^{i-n}\tbinom{i-1}{n} \,, \\
	\ga_2(n,i) & = \begin{cases} 0 &\text{ if } n=0 \,, \\
	(-1)^{i-n-1}\tbinom{i-2}{n-1} &\text{ if } n>0  \,, 
	\end{cases} \\
	\ga_3(r,n,i) & = \begin{cases} 0 &\text{ if } i=0 \,, \\
	(-1)^{r-1-(n-i)}\tbinom{i-1}{n-r} &\text{ if } i >0 \,, 
	\end{cases} \\
	\ga_4(r,n,i) & = \begin{cases} 0 &\text{ if $ i\le1$ and $n>r$ } \,, \\
	   r-1 & \text{ if $i=1$ and $n=r$ } \,, \\
	   (-1)^{r-1-(n-i)} \frac{r i -n}{i-1} \tbinom{i-1}{n-r} &\text{ if $i>1$ }\,.
	\end{cases}
\end{align}
\end{subequations}
In order to simplify \eqref{sum_j}, we will have to determine the sum
\begin{align}
	&\sum_{i=0}^{r-1-k} \tbinom{i+k}{k} \sum_{n=0}^{r-1} \ze^n r \tbinom{r}{i+k+1} \ga_1(n,i) \notag \\
	&\quad = \sum_{n=0}^{r-1} \ze^n r \sum_{m=0}^{r-1-k} \tbinom{m+k}{k} \tbinom{r}{m+k+1} \ga_1(n,m) \,,
\end{align}
and analogously for $\ga_2$\,.

For the corresponding sums in \eqref{sum_j} with $\ga_3$ and $\ga_4$, which run from $n=r$ to $n=r-1+i$, we use 
\begin{equation}
	\sum_{i=0}^I a_i \sum_{n=N}^{N-1+i} \ze^n b_{n,i} = \sum_{n=N}^{N+I} \ze^n \sum_{m=n-N+1}^I a_m b_{n,m} \,,
\end{equation}
where $I=r-1-k$ and $N=r$. After adding the appropriate terms and simple rearrangement, we obtain the coefficients $\cgt(r,k,n)$ of $u^k\ze^n$ in the form \eqref{cgt_all}.
It follows that for given $r$, $k$, and $n$ the coefficients $c_{\tilde g_i}(r,k,n,m)$ alternate in sign as $m$ increases. 

Now we are ready to derive the representations of $\cgt(r,k,n)$ obtained in equations \eqref{cgt_final}. 
It is useful to keep the following in mind:
\begin{equation*}
	\binom{m-2}{n-1} = \begin{cases} (-1)^{n-1}n &\text{ if $m = 0$ and $n\ge0$}  \,, \\
				(-1)^{n-1} &\text{ if $m = 1$ and $n\ge1$}  \,, \\
				0 &\text{ if } n\ge m\ge 2 \,, \\
				1 &\text{ if  $n = 1$ or $n=m-1$} \,, \\
				>1 &\text{ if  $m \ge n+2$ and $n\ge2$} \,.	\end{cases}
\end{equation*}

STEP 2. To show \eqref{cgt_final_a}, we assume $0\le n \le r -k-1$. The starting point is \eqref{cgt_all}.Then
\begin{subequations}
\begin{equation}
	\sum_{m=0}^{r-1-k} \tbinom{m+k}{k} \tbinom{r}{m+k+1} (-1)^{m-n} \tbinom{m-1}{n} =  \tbinom{k+n+1}{k+1} =  \tbinom{k+n+2}{k+2} \frac{k+2}{k+n+2} \,,
\end{equation}
\begin{equation}
	\sum_{m=0}^{r-1-k} \tbinom{m+k}{k} \tbinom{r+1}{m+k+2} (-1)^{m-n} \tbinom{m-1}{n} =  \tbinom{k+n+2}{k+2} \frac{r(k+2) -k(k+n+2)-n }{k+n+2} \,,
\end{equation}
and
\begin{equation}\label{third_term}
	\sum_{m=0}^{r-1-k} \tbinom{m+k}{k} \tbinom{r}{m+k+1} (-1)^{m-n} \tbinom{m-2}{n-1} = \tbinom{k+n+2}{k+2} \frac{-n}{k+n+2} \,,
\end{equation}
\end{subequations}
where the latter holds only if $n\ge1$.
Applying these identities to $\cgt$ in \eqref{cgt} with $c_{\tilde g_1}$ in \eqref{cgt1}, and observing that $\tbinom{r+1}{m+k+2} = \tbinom{r}{m+k+1}\frac{r+1}{m+k+2}$,
we obtain the desired result:
\begin{align*}
	\cgt(r,k,n) &= \tbinom{k+n+2}{k+2}\Bigl[r \frac{k+2}{k+n+2} - \frac{r(k+2) -k(k+n+2)-n }{k+n+2} -  \frac{n}{k+n+2} \Bigr] \\
		&= k\tbinom{k+n+2}{k+2}\,.
\end{align*}
If $n=0$, the term in \eqref{third_term} is absent and $\cgt(r,k,0)=k$ follows.

STEP 3. Next we derive \eqref{cgt_final_b} and assume $r-k\le n \le r-1$. We start with the first case in \eqref{cgt} and the second in \eqref{cgt1}.
Because $\binom{m-1}{n}=0$ if $1\le m\le n$, we have $\sum\limits_{m=0}^{r-1-k} a(m,n) \binom{m-1}{n} = \sum\limits_{m=r-k+1}^{r-1-k} a(m,n) \binom{m-1}{n} = 0$ in this case. Therefore, only the summand with $m=0$ remains (because $n\ge 1$ whence $\binom{0}{n}=0$) and we obtain
\begin{subequations}
\begin{equation}\label{cg1_decrease_a}
	\sum_{m=0}^{r-1-k} \tbinom{m+k}{k} \tbinom{r}{m+k+1} (-1)^{m-n} \tbinom{m-1}{n}r  = r \tbinom{r}{k+1} (-1)^{-n} \tbinom{-1}{n} = r \tbinom{r}{k+1} \,.
\end{equation}
Analogously, and including the factor $\frac{r+1}{m+k+2}$ from \eqref{cgt1}, we find
\begin{equation}\label{cg1_decrease_b}
	\sum_{m=0}^{r-1-k} \tbinom{m+k}{k} \tbinom{r+1}{m+k+2} (-1)^{m-n} \tbinom{m-1}{n} = \tbinom{r+1}{k+2} (-1)^{-n} \tbinom{-1}{n} = \tbinom{r}{k+1}\frac{r+1}{k+2} \,.
\end{equation}

Because $\binom{m-2}{n-1}=0$ if $2\le m\le n$, we have $\sum\limits_{m=0}^{r-1-k} a(m,n) \binom{m-2}{n-1} = \sum\limits_{m=r-2-k}^{r-1-k} a(m,n) \binom{m-2}{n-1} = 0$ in this case. Therefore, only summands with $m=0$ or $m=1$ remain and we obtain
\begin{align}\label{cg1_decrease_c}
	&\sum_{m=0}^{r-1-k} \tbinom{m+k}{k} \tbinom{r}{m+k+1} (-1)^{m-n} \tbinom{m-2}{n-1} \notag \\
	&\qquad = 1\cdot\tbinom{r}{k+1} (-1)^{-n} \tbinom{-2}{n-1} + \tbinom{1+k}{k} \tbinom{r}{k+2} (-1)^{1-n}  \tbinom{-1}{n-1} \notag \\
	&\qquad = \tbinom{r}{k+1} (-1) n + (k+1)\tbinom{r}{k+2} \notag \\	
	&\qquad= \tbinom{r}{k+1} \Bigl(\frac{k+1}{k+2}(r-1-k)-n \Bigr)  \,.
\end{align}
\end{subequations}

Adding \eqref{cg1_decrease_a}, \eqref{cg1_decrease_b}, and \eqref{cg1_decrease_b} according to \eqref{cgt1} we obtain
\begin{equation}
	\cgt(r,k,n) = \binom{r}{k+1}\frac{(2r-k)(k+1)-(n+1)(k+2)}{k+2}
\end{equation}
if $r-k \le n \le r-1$.

STEP 4. Finally, we deduce \eqref{cgt_final_c} and assume $r \le n \le 2r-3-k$. We assume $r+1\le n \le 2r-3-k$ and set $s=n-r$. (We leave the verification of the simpler case $n=r$ to the reader.) We need to simplify the sums resulting from case two in \eqref{cgt} combined with \eqref{cgt1}. For given $n\ge r+1$ (so that $1\le s \le r-3-k$), the sum invoking $\binom{m-1}{n-r}$ becomes
\begin{subequations}
\begin{align}\label{term2_a}
	&\sum_{m=s+1}^{r-1-k} \tbinom{m+k}{k} \tbinom{r+1}{m+k+2} (-1)^{m-s} \tbinom{m-1}{s} \notag \\
	&\quad = -\tbinom{r+1}{k+2} + \Bigl(\frac{(k+2)(r-k)}{s}-(k+1)\Bigr)\tbinom{k+s+1}{k+2} \notag \\
	&\quad = -\tbinom{r}{k+1}\frac{r+1}{k+2} + \Bigl((r-k)-\frac{(k+1)s}{k+2}\Bigr)\tbinom{k+s+1}{k+1} \,.
\end{align}
The sum invoking $\binom{m-2}{n-r-1}$ becomes
\begin{align}\label{term2_b}
	&\sum_{m=s+1}^{r-1-k} \tbinom{m+k}{k} \tbinom{r}{m+k+1} (-1)^{m-s} \tbinom{m-2}{s-1} \notag \\
	&\quad = \tbinom{r}{k+2}\Bigl(\frac{s(k+2)}{r-k-1} - (k+1)  \Bigr) - \tbinom{k+s+1}{k+2} \notag \\
	&\quad = \tbinom{r}{k+1} \Bigl(s - \frac{(r-k-1)(k+1)}{k+2} \Bigr) - \tbinom{k+s+1}{k+1}\frac{s}{k+2} \,.
\end{align}
\end{subequations}
Substracting \eqref{term2_b} from \eqref{term2_a} and returning to $n=r+s$, we arrive
at \eqref{cgt_final_c}, which finishes the proof of the lemma.
\end{proof}

\vspace{0.5cm}
\noindent{\bf Additional materials} A \emph{Mathematica} notebook containing the code to check the algebraically demanding parts of the proof is made available upon request.

\singlespacing 

\end{document}